\documentclass{article}
\usepackage{amsmath, amsthm, amsfonts}
\usepackage[utf8]{inputenc}
\usepackage{srcltx}

\usepackage[colorlinks=true,linkcolor=blue,citecolor=blue]{hyperref}

\newcommand{\iO}{\int_{\Omega}}
\newcommand{\iOT}{\int_{\Omega_T}}
\newcommand{\eps}{\varepsilon}

\def\N{\mathbb{N}}
\def\R{\mathbb{R}}

\def\pa{\partial}
\def\var{\varepsilon}
\newtheorem{thm}{Theorem}[section]

\newtheorem{lem}[thm]{Lemma}
\newtheorem{prp}[thm]{Proposition}

\theoremstyle{definition}

\theoremstyle{remark}
\newtheorem{rem}[thm]{Remark}

\title{Improved duality estimates and applications to
  reaction-diffusion equations}

\date{April 13, 2013}

\author{Jos\'e A. Ca\~nizo\footnote{School of Mathematics, University
    of Birmingham, Edgbaston, Birmingham B15 2TT, UK. Email:
    \texttt{j.a.canizo@bham.ac.uk}. Supported by the project MINECO
    MTM2011-27739-C04-02 and the Marie-Curie CIG project KineticCF.}
  \and Laurent Desvillettes\footnote{CMLA, ENS Cachan, CNRS, 
    61 Av. du Pdt. Wilson, 94235 Cachan Cedex,
    France. Email: \texttt{desville@cmla.ens-cachan.fr}}
  \and Klemens Fellner\footnote{Institut f\"ur Mathematik und
    Wissenschaftliches Rechnen, Heinrichstra{\ss}e 36, 8010
    Graz. Email: \texttt{klemens.fellner@uni-graz.at}}
}

\begin{document}

\maketitle

\begin{abstract}
  We present a refined duality estimate for parabolic equations. This
  estimate entails new results for systems of reaction-diffusion
  equations, including smoothness and exponential convergence towards
  equilibrium for equations with quadratic right-hand sides in two
  dimensions. For general systems in any space dimension, we obtain
  smooth solutions of reaction-diffusion systems coming out of
  reversible chemistry under an assumption that the diffusion
  coefficients are sufficiently close one to another.
\end{abstract}

\section{Introduction}
\label{sec:intro}

This paper presents a refined duality estimate for reaction-diffusion
equations arising in the context of reversible chemistry, of the form
\begin{equation}\label{genpol}
  \pa_t a_i - d_i\,\Delta_x a_i = (\beta_i-\alpha_i)
  \left(l \,
    \prod\limits_{j=1}^n a_j^{\alpha_j}
    - k\, \prod\limits_{j=1}^n a_j^{\beta_j}
    \, \right), \qquad i=1..n,
\end{equation}
with the homogeneous Neumann boundary conditions
\begin{equation}
  \label{CN}
  \nabla_x a_i(t,x) \cdot \nu(x) = 0
  \quad
  \text{ for } x \in \pa\Omega,\ t \geq 0
\end{equation}
corresponding to the diffusion of $n$ species ${\mathcal{A}}_i$ with
concentration $a_i := a_i(t,x)\ge 0$, $i=1..n$ at time $t \geq 0$ and
point $x \in \Omega \subset \R^N$, each with its own diffusion
coefficient $d_i\ge 0$, and to the reversible chemical reaction
\begin{equation}\label{un}
\alpha_1\mathcal{A}_1 + \dots + \alpha_n\mathcal{A}_n \rightleftharpoons
\beta_1\mathcal{A}_1 + \dots + \beta_n\mathcal{A}_n,\,\quad \alpha_i,
\beta_i \in \mathbb{N},
\end{equation}
where the above reaction is modelled according to the \emph{mass
  action law} with the stoichiometric coefficients $\alpha_i, \beta_i
\in \mathbb{N}$ and with the (constant) reaction rates $l,k >0$. The
mixture is assumed to be confined in a domain $\Omega \subset \R^N$ as
implied by the homogeneous Neumann condition (\ref{CN}), where
$\nu(x)$ denotes the outward normal vector to $\Omega$ at point $x\in
\pa\Omega$.
\par
Moreover, we systematically denote by $\Omega_T = [0,T] \times \Omega$
(for any given $T>0$) and by $p'$ the H\"older conjugate exponent of
$p$, i.e. $\frac{1}{p} + \frac{1}{p'} = 1$.

\bigskip

The mathematical difficulties in proving existence, smoothness and
large-time behaviour theories for systems like (\ref{genpol}),
(\ref{CN}) increase with the degree of the polynomials terms
(appearing in the r.h.s. of (\ref{genpol})) as well as with the number
$n$ of equations or the space dimension $N$.  \bigskip

The refined duality estimate that we shall derive in this paper
depends yet on another parameter of (\ref{genpol}), namely the maximal
distance between the diffusion rates appearing in (\ref{genpol}), that
is
\begin{equation}
 \delta := \sup_{i= 1..n} \{d_i\}  - \inf_{i= 1..n} \{d_i\}.
\end{equation}
It is easy to see that when $\delta=0$, the system (\ref{genpol}),
(\ref{CN}) can be rewritten as the coupled system between $n-1$ heat
equations (for the sums of two concentrations) and a single
reaction-diffusion equation, which greatly simplifies the analysis
compared to the general case when $\delta >0$. In particular, the
dynamics of system (\ref{genpol}) for $\delta=0$ satisfies a maximum
principle, which fails to be true for general diffusion systems with
$\delta >0$ (except for special systems where the structure of the
reaction terms enforces a maximum principle).
\par 
Perturbation methods can sometimes be used in order to transfer at
least partly the properties of the system with $\delta=0$ to the case
when $\delta>0$ is small, see e.g. \cite{HM}.
\par
Our new estimate is also particularly efficient in the case when
$\delta>0$ is small, but it still gives results for all parameters
$\delta$ in some situations, and even when for example one of the
diffusion rates is $0$. Moreover, when the smallness of $\delta$ is
required, it can be explicitly estimated.  \bigskip

In order to obtain this estimate, we use a combination of ideas coming
from maximal elliptic regularity, a Meyers-type estimate which
provides an explicit perturbation argument, and duality methods in the
line of e.g. \cite{HMP,HM,PS00}. We end up with the following
Proposition for parabolic equations with variable coefficients:

\begin{prp}
  \label{propun} 
  Let $\Omega$ be a bounded domain of $\R^N$ with smooth
  (e.g. $C^{2+\alpha}$, $\alpha>0$) boundary $\partial\Omega$, $T>0$,
  and $p\in ]2,+\infty[$. We consider a coefficient function $M :=
  M(t,x)$ satisfying
  \begin{equation}
    \label{eq:M-bounds}
    0 < a \leq M(t,x) \leq b < +\infty
    \quad \text{ for } (t,x) \in \Omega_T,
  \end{equation}
  for some $0 < a < b < +\infty$, and an initial datum $u_0 \in
  L^p(\Omega)$.  \medskip

  Then, any weak solution $u$ of the parabolic system:
  \begin{equation}
    \label{eq:heat-variable-forward}
    \left\{
      \begin{aligned}
        &\partial_t u - \Delta_x (M u) = 0 &&\quad \text{ on } \quad \Omega_T,
        \\
        &u(0,x) = u_0(x) &&\quad \text{ for } \quad x \in \Omega,
        \\
        &\nabla_x u \cdot \nu(x) = 0 &&\quad \text{ on } \quad [0,T]
        \times \partial \Omega ,
      \end{aligned}
    \right.
  \end{equation}
  satisfies the estimate (where $p'<2$ denotes the H\"older conjugate
  exponent of $p$)
  \begin{equation}
    \label{eq:heat-estimate-forward}
    \|u\|_{L^p(\Omega_T)}
    \leq
    (1 + b \,D_{a,b,p'})\, T^{1/p} \,
    \|u_0\|_{L^p(\Omega)},
  \end{equation}
  and where for any $a,b>0$, $q\in ]1,2[$
  \begin{equation}
    \label{dixsept}
    D_{a,b,q}
    :=
    \frac{C_{\frac{a+b}2,q}}{1 - C_{\frac{a+b}2,q} \frac{b-a}{2}},
  \end{equation}
  provided that the {following condition holds}
  \begin{equation}
    \label{eq:const-small-q}
    C_{\frac{a+b}2,p'} \, \frac{b-a}{2} < 1.
  \end{equation}
  Here, the constant $C_{m,q}>0$ is defined for $m>0$, $q \in ]1,2[$
  as the best (that is, smallest) constant in the parabolic regularity
  estimate
  \begin{equation}
    \label{huit}
    \| \Delta_x v\|_{L^q(\Omega_T)} \le C_{m,q}\,  \|f\|_{L^q(\Omega_T)},
  \end{equation} 
  where  $v:[0,T] \times \Omega \to \R$ is the solution
  of the backward heat equation with homogeneous Neumann boundary conditions:
  \begin{equation}
    \label{quinze}
    \left\{
      \begin{aligned}
        &\partial_t v + m \,\Delta_x v = f &&\quad \text{ on } \quad \Omega_T,
        \\
        &v(T,x) = 0 &&\quad \text{ for } \quad x \in \Omega,
        \\
        &\nabla_x v \cdot \nu (x) = 0 &&\quad \text{ on } \quad  [0,T]
        \times \partial \Omega.
      \end{aligned}
    \right.
  \end{equation}

  We recall that one has $C_{m,q} < \infty$ for $m>0$, $q \in ]1,2[$
  and in particular $C_{m,2} \le\frac{1}{m}$ (cf. Lemma~\ref{lem21}
  below).  Note that the constant $C_{m,q}$ may depend (besides on $m$
  and $q$) also on the domain $\Omega$ and the space dimension $N$,
  but {does not depend on the time $T$}.
\end{prp}

The consequences of Proposition \ref{propun} can be best understood in
the case of the most standard reversible chemical reaction, that is
when (\ref{un}) writes
\begin{equation}
    \label{a1a4}
\mathcal{A}_1 + \mathcal{A}_3 \rightleftharpoons \mathcal{A}_2 + \mathcal{A}_4
\end{equation}
and (\ref{genpol}) becomes (after the rescaling of the nonessential constants $k$ and $l$ to unity)
\begin{equation}
    \label{trois}
    \left\{
      \begin{aligned}
\pa_t a_1 - d_1 \, \Delta_x a_1 &= (a_4\,a_2 - a_1\, a_3),\\
\pa_t a_2 - d_2 \, \Delta_x a_2 &= -(a_4\,a_2 - a_1\, a_3),\\
\pa_t a_3 - d_3 \, \Delta_x a_3 &=  (a_4\,a_2 - a_1\, a_3),\\
\pa_t a_4 - d_4 \, \Delta_x a_4 &=  -(a_4\,a_2 - a_1\, a_3).
    \end{aligned}
    \right.
\end{equation}
We recall that for this system (with the boundary conditions
(\ref{CN}) and provided that $d_i>0$ for $i=1..4$), existence of weak
solutions in $L^2 (\log L)^2$ was obtained in \cite{DFPV}, together
with the existence of strong (smooth) solutions when $N=1$. This
result was improved by Th. Goudon and A. Vasseur in \cite{GV} thanks
to a careful use of De Giorgi's method, see e.g. \cite{DeG}. They
showed that strong solutions also exist when $N=2$. We also refer to
\cite{CV}, where smooth solutions were shown to exist in any dimension
for systems with a nonlinearity of power law type which is strictly
subquadratic, see also e.g. \cite{Ama}.  \bigskip

We also recall two results on exponential convergence to the
equilibrium: First, exponential convergence in any $H^p$ norm in the
one-dimensional case $N=1$ was obtained for the system (\ref{trois})
with boundary conditions (\ref{CN}) in \cite{DF08}. This result is
based on the use of the entropy/entropy dissipation method with slowly
growing a priori $L^{\infty}$-bounds (cf. \cite{TVillani}, \cite{DM}).
It required ``at most polynomially growing w.r.t. $T$'' bounds for the
quantities $\sup_{t \in [0,T]} \|
a_i(t,\cdot)\|_{L^{\infty}(\Omega)}$. A related, yet non-explicit
approach to entropy methods for reaction-diffusion type systems can be
found e.g. in \cite{Gro92,GGH96}.

In a later improvement \cite{DFEqua}, the authors showed exponential
convergence to equilibrium in relative entropy avoiding any
$L^{\infty}$-bounds on the solution. Thus, interpolating the weak
global $L^2$ solutions constructed in \cite{DFPV}, one obtains
exponential convergence towards equilibrium in any $L^p$ norm with
$p<2$ for all space dimension $N>1$.

It is interesting to point out that for space dimensions $N\ge3$ the
existence of global classical solutions for general (even constant)
diffusion coefficients and initial data is an open problem despite the
fact that all $L^2$ solutions converge exponentially towards the
constant equilibrium in $L^p$ with $p<2$. Up to our knowledge, it is
only known that if solutions to (\ref{trois}) with boundary conditions
(\ref{CN}) would blow-up in the $L^{\infty}$ norm, then such a
concentration phenomenon would need to occur in at least two densities
$a_i$ at the same position $x_0\in\Omega$ at the same time $t_0>0$
\cite{HM}.  Moreover, an upper bound on the Hausdorff dimension of
singularities was given in \cite{GV}.  \bigskip

Thanks to Proposition \ref{propun}, this paper is able to provide a
direct proof of the result in \cite{GV} (that is, without use of De
Giorgi's method) when $N=2$.  Moreover, we can obtain immediately the
exponential convergence of the solution of (\ref{trois}), (\ref{CN})
towards equilibrium in $L^{\infty}$, which is a significant
improvement on the above mentioned $L^{2-0}$-convergence of
\cite{DFEqua}. It is remarkable that in this specific case, \emph{no
  smallness requirement for $\delta$} appears in the assumptions.
More precisely, we prove the

\begin{prp}\label{propdeux}
  Let $\Omega$ be a bounded domain of $\R^2$ with smooth
  (e.g. $C^{2+\alpha}$, $\alpha>0$) boundary $\partial\Omega$. For all
  $[i=1..4]$ assume positive diffusion coefficients $d_i>0$ and
  nonnegative initial data $a_{i0} \in L^{\infty}(\Omega)$.  \medskip

  Then, there exists a weak nonnegative solution $a_i \in
  L^{\infty}([0, +\infty[ \times \Omega)$ to the system
  (\ref{trois}) with homogenenous Neumann boundary conditions
  (\ref{CN}) subject to the initial data $a_{i0}$ for all $[i=1..4]$.
  \medskip

  Moreover, we denote for $[i=1..4]$ by $a_{i\infty} >0$ the
  equilibrium values of the concentrations $a_i$: Thus,
  $\{a_{i\infty}\}_{i=1..4}$ is the unique vector of positive
  constants balancing the reaction rate
  $$a_{1\infty}\, a_{3\infty} = a_{2\infty}\, a_{4\infty}
  $$ 
  and satisfying the three (linear independent) mass-conservation laws 
  \begin{align*}
    a_{1\infty} + a_{2\infty} &= \frac{1}{|\Omega|}\iO (a_{10} + a_{20})\,dx,\\
    a_{1\infty} + a_{4\infty} &= \frac{1}{|\Omega|}\iO (a_{10} + a_{40})\,dx,\\
    a_{2\infty} + a_{3\infty} &= \frac{1}{|\Omega|}\iO (a_{20} + a_{30})\,dx.
  \end{align*}
  Then, there exist two constants $\kappa_1, \kappa_2>0$ such that
  \begin{equation}\label{quatre}
    \forall t \ge 0, \qquad \qquad \sum_{i=1}^4 \| a_i(t,\cdot) - a_{i\infty}\|_{L^{\infty}(\Omega)} \le \kappa_1\, e^{- \kappa_2\, t} .
  \end{equation}
  \par
  Moreover, the norm $\|a_i\|_{L^{\infty}([0, +\infty[ \times\Omega)}$
  and the constants $\kappa_1, \kappa_2$ can be explicitly bounded in
  terms of the domain $\Omega$, space dimension $N$, the norm
  $\|a_{i0}\|_{L^{\infty}(\Omega)}$ of the initial data and the
  diffusion coefficients $d_i$, $[i=1..4]$.
\end{prp}
\bigskip

Proposition \ref{propun} also entails that a similar result to
Proposition \ref{propdeux} holds in any space dimension provided that
$\delta>0$ is small enough. More precisely, we prove the:

\begin{prp}\label{proptrois}
  Let $\Omega$ be a bounded domain of $\R^N$ with smooth
  (e.g. $C^{2+\alpha}$, $\alpha>0$) boundary $\partial\Omega$.  For
  all $[i=1..4]$ assume positive diffusion coefficients $d_i>0$ such
  that $0<a = \inf \{d_i\}_{i=1..4}$, $0<b=\sup \{d_i\}_{i=1..4}$ and
  nonnegative initial data $a_{i0} \in L^{\infty}(\Omega)$ ($i=1..4$).
  \medskip

  Then, if $\delta= b-a < 2\,(C_{\frac{a+b}2, 1 + 2/N})^{-1}$, there
  exists a nonnegative weak solution $a_i \in L^{\infty}([0, +\infty[
  \times \Omega)$ to the system
  (\ref{trois}) with homogeneous Neumann boundary conditions
  (\ref{CN}) subject to the initial data $a_{i0}$.
  \medskip

  Moreover (with the notation of the previous Proposition \ref
  {propdeux}), there exist two constants $\kappa_1, \kappa_2>0$ such
  that
  \begin{equation}\label{quatrebis}
    \forall t \ge 0, \qquad \qquad
    \sum_{i=1}^4 \| a_i(t,\cdot) - a_{i\infty}\|_{L^{\infty}(\Omega)}
    \le \kappa_1\, e^{- \kappa_2\, t}.
  \end{equation}
  \par
  Here, the norm $\|a_i\|_{L^{\infty}([0, +\infty[ \times\Omega)}$ and
  the constants $\kappa_1, \kappa_2$ can be explicitly bounded in
  terms of the domain $\Omega$, space dimension $N$, the norm
  $\|a_{i0}\|_{L^{\infty}(\Omega)}$ of the initial data, and the
  diffusion coefficients $d_i$, $(i=1..4)$.
\end{prp}
\bigskip

\bigskip

A further result based on Proposition \ref{propun} states an existence
theorem for weak and (with more stringent assumption) bounded weak
solutions of (\ref{genpol}), (\ref{CN}) when the r.h.s. of
(\ref{genpol}) is not necessarily quadratic anymore. An assumption
about the smallness of $\delta>0$ is still needed here.

\begin{prp}\label{propquatre}
  Let $\Omega$ be a bounded domain of $\R^N$ with smooth
  (e.g. $C^{2+\alpha}$, $\alpha>0$) boundary $\partial\Omega$.  For
  all $i=1..n$ assume positive diffusion coefficients $d_i>0$ such
  that $0<a = \inf \{d_i\}_{i=1..n}$, $0<b=\sup \{d_i\}_{i=1..n}$ and
  nonnegative initial data $a_{i0} \in L^{\infty}(\Omega)$.  Moreover,
  let $k,l>0$, $\alpha_i, \beta_i \in \N$ be such that at least two
  coefficients $\beta_i - \alpha_i$ are different from $0$ and have
  opposite signs. We define $Q= \sup\{ \sum_{i=1}^n \alpha_i,
  \sum_{i=1}^n \beta_i\}$ and assume that $Q\ge 3$.  \medskip

\begin{description}
 \item Then, if $\delta= b-a < 2\,(C_{\frac{a+b}2, Q'})^{-1}$ (where $Q'$ is the H\"older conjugate of $Q$), there exists a nonnegative weak solution  
$a_i \in L^{Q}(\Omega_T)$ for all $T>0$ to the system
(\ref{genpol}), (\ref{CN}) with the initial data $a_{i0}$. 
\item Moreover, if $\delta= b-a < 2\,(C_{\frac{a+b}2, \frac{(Q-1)\,(N+2)}{(Q-1)\,(N+2) - 2} })^{-1}$,
then the solution $a_i$ lies in $L^{\infty}(\Omega_T)$ for all $T>0$.
\end{description}
\end{prp}

Finally, we show that there exist bounded weak solutions to the system (\ref{trois}) with homogeneous Neumann boundary conditions (\ref{CN}) 
in space dimension~2 even when one of the diffusion rates (say $d_4$ w.l.o.g.) is equal to $0$.
We remark that existence of global weak solutions
in the case of degenerate diffusion was also shown in 
\cite{DFEqua}.

\begin{prp}\label{propcinq}
Let $\Omega$ be a bounded domain of $\R^2$ with smooth (e.g. $C^{2+\alpha}$, $\alpha>0$) boundary $\partial\Omega$. For all $[i=1..4]$ assume nonnegative initial data $a_{i0} \in L^{\infty}(\Omega)$ and nonnegative diffusion coefficients $d_i\ge0$ with e.g. $d_i>0$, $[i=1..3]$ and $d_4=0$.
\medskip

Then, there exists a nonnegative weak solution  $a_i$ lying in $L^{\infty}(\Omega_T)$
for all $T>0$ to the system
(\ref{trois}), (\ref{CN}) (without the Neumann boundary condition on $a_4$) subject to the 
initial data $a_{i0}$.
\end{prp}
\bigskip

\begin{rem}
The assumption that the initial data lie in $L^{\infty}$ enables to give a simple formulation
of the Propositions above, but it is not optimal (if one is only interested in the bounds 
of the solutions after a given positive time $t_0>0$). For example, in the case of Proposition
\ref{propdeux}, it is easy to see that if the initial data lie in $L^p(\Omega)$ for some
$p>2$, then the conclusion remains true with the time interval $[0, +\infty[$ changed into $[t_0, +\infty[$
(for any $t_0>0$) for the bounds. A more careful analysis (cf. Remark \ref{rein}) shows that the assumption on the initial data
can even be relaxed to $L^p(\Omega)$ for some
$p>1$.
\end{rem}
\medskip

\begin{rem}
Classical bootstrap arguments also show that all the above weak solutions, once they belong to 
$L^{\infty}(Q_T)$, are in fact strong and smooth provided that the 
set $\Omega$ has a smooth enough boundary (and provided that the initial datum is also smooth enough, if one wishes to get smoothness even at point $t=0$).
Moreover in this case, those solutions are unique (in the set of smooth enough solutions)
and (in the case of Propositions \ref{propdeux} and \ref{proptrois}) converge towards equilibrium exponentially fast (with explicitly computable constants) in any $H^p$ norm
($p\in \N$), thanks to interpolation arguments similar to those exposed in \cite{DF08}.
\end{rem}
\medskip
\begin{rem}
  The smoothness assumption $C^{2+\alpha}$, $\alpha>0$ on the boundary
  $\partial\Omega$ is likely not optimal. One could conjecture that
  the above results hold true also for $C^{1+\alpha}$ boundaries or
  even Lipschitz boundaries. However, for more general boundaries, we
  lack, for instance, a reference which states explicitly the
  time-independence of the constant $C_{m,q}$ in \eqref{huit}.
\end{rem}
\medskip

The paper is organized as follows: In Section \ref{secdeux}, we present the Proof of Proposition
\ref{propun}. Section \ref{sectrois} is devoted to the applications of Proposition \ref{propun}
to the ``four species''
system \eqref{trois}, first in dimension~2 (Proof of Proposition \ref{propdeux}) and then in any dimension (Proof of Proposition \ref{proptrois}). 
Finally, we present in Section~\ref{secquatre} the extensions
to more general reaction-diffusion systems (Proof of Proposition~\ref{propquatre}) and to the case when one diffusion rate is $0$ (Proof of Proposition \ref{propcinq}).

\section{An estimate for singular parabolic problems}
\label{secdeux}

We first recall a well-known result for the heat equation, which
ensures that for $m>0$, $p \in ]1, 2[$ the constants $C_{m,p}$ stated
in eq. (\ref{huit}) are well-defined, finite and time-independent.

\begin{lem}
  \label{lem21}
  Let $\Omega$ be a bounded domain of $\R^N$ with smooth
  (e.g. $C^{2+\alpha}$, $\alpha>0$) boundary $\partial\Omega$, $m>0$,
  and $p \in ]1,2]$.  Then, there exists a constant $C_{m,p} > 0$
  depending on $m$, $p$, the domain $\Omega$ and the space dimension
  $N$, but not on $T$, such that the solution $v:[0,T] \times \Omega
  \to \R$ of the backward heat equation with homogeneous Neumann
  boundary conditions (\ref{quinze}) satisfies
  \begin{equation}
    \label{eq:heat-estimate-p}
    \|\Delta_x v\|_{L^p(\Omega_T)}
    \leq
    C_{m, p}\, \|f\|_{L^p(\Omega_T)},
  \end{equation}
  where $f \in L^p(\Omega_T)$ is the r.h.s. of (\ref{quinze}). 
  Moreover, $C_{m,2} \le 1/m$.
\end{lem}

\begin{proof}[Proof of Lemma \ref{lem21}] After introducing the time
  variable $\tau= T-t\in[0,T]$, the backward heat equation
  \eqref{quinze} with Neumann boundary conditions and zero end data
  transforms into the forward heat equation with zero initial data:
  \begin{equation*} 
    \left\{
      \begin{aligned}
        &\partial_{\tau} v - m \,\Delta_x v = - f &&\quad \text{ on } \quad \Omega_T,
        \\
        &v(0,x) = 0 &&\quad \text{ for } \quad x \in \Omega,
        \\
        &\nabla_x v \cdot \nu (x) = 0 &&\quad \text{ on } \quad  [0,T]
        \times \partial \Omega.
      \end{aligned}
    \right.
  \end{equation*}
  Moreover, the semigroup of the forward heat equation with
  homogeneous Neumann boundary condition satisfies the contraction
  property $\|e^{t \Delta_x} v(0,\cdot)\|_p\le\|v(0,\cdot)\|_p$ for
  all $p\in[1,\infty]$ and for all $t\ge0$. Thus, the statement of the
  Lemma follows from \cite{JFA}, where it is explicitly stated that
  $C_{m,p}$ can be taken as time-independent. In particular, the
  Hilbert space case $p=2$ allows explicit calculations by testing the
  above forward heat equation with $-\Delta_x$, which shows that
  $C_{m,2} \le\frac{1}{m}$, see eq. \eqref{eq:heat-estimate-p=2}
  below.
\end{proof}
 
We now consider the corresponding problem with variable diffusion rate
and obtain similar estimates:

\begin{lem}
  \label{lem23}
Let $\Omega$ be a bounded domain of $\R^N$ with smooth (e.g. $C^{2+\alpha}$, $\alpha>0$) boundary $\partial\Omega$, $T > 0$, $p \in ]1,2]$, 
  and $M := M(t,x)$ be  bounded above and below; i.-e.
 for some $0<a\le b $, 
  \begin{equation}
    0 < a \leq M(t,x) \leq b < +\infty
    \quad \text{ for } (t,x) \in \Omega_T.
  \end{equation}
  We assume (using the notation of Lemma \ref{lem21}) that
\begin{equation}
    \label{eq:const-small}
    C_{\frac{a+b}2,p}\, \,\frac{b-a}{2} < 1.
  \end{equation}

  We consider $f\in L^p(\Omega_T)$, and a solution $v$ of the parabolic equation with
  variable diffusion rate given by $M$:
  \begin{equation}
    \label{eq:heat-variable}
    \left\{
      \begin{aligned}
        &\partial_t v + M \,\Delta_x v = f &&\quad \text{ on } \quad \Omega_T,
        \\
        &v(T,x) = 0 &&\quad \text{ for } \quad x \in \Omega,
        \\
        &\nabla_x v \cdot \nu(x) = 0 &&\quad \text{ on } \quad [0,T]
        \times \partial \Omega.
      \end{aligned}
    \right.
  \end{equation}
  Then,
  \begin{equation}
    \label{eq:heat-estimate-variable}
    \|\Delta_x v\|_{L^p(\Omega_T)}
    \leq
    D_{a,b,p}\,
    \| f \|_{L^p(\Omega_T)},
  \end{equation}
and 
 \begin{equation}
    \label{eq:v0-bound}
    \| v(0,\cdot) \|_{L^p(\Omega)}
    \leq (1 + b \,D_{a,b,p})\, T^{1/p'} \,\|f\|_{L^p(\Omega_T)},
  \end{equation}
  where $D_{a,b,p}$ is given by (\ref{dixsept}), i.e.
  $$    
  D_{a,b,p} :=
  \frac{C_{\frac{a+b}2,p}}{1 - C_{\frac{a+b}2,p} \frac{b-a}{2}},
  $$
\end{lem}

\begin{proof}[Proof of Lemma \ref{lem23}]
  In order to find estimates analogous to \eqref{eq:heat-estimate-p}
for our variable coefficients parabolic equation, 
we take $m := (a + b)/2$ and rewrite \eqref{eq:heat-variable} as the perturbative problem
  \begin{equation}
    \label{eq:heat-variable-2}
    \partial_t v +  m \,\Delta_x v =
    \left(m - M \right) \, \Delta_x v + f.
  \end{equation}
  Then from \eqref{eq:heat-estimate-p}, we get
  \begin{align}
    \|\Delta_x v\|_{L^p(\Omega_T)} &\leq C_{m,p}\, \left\| \left(m - M
      \right) \Delta_x v + f \right\|_{L^p(\Omega_T)}\nonumber 
    \\
    \label{eq:variable-estimate-1}
    &\leq C_{m,p}\, \left( \frac{b-a}{2} \left\| \Delta_x v
      \right\|_{L^p(\Omega_T)} + \left\| f \right\|_{L^p(\Omega_T)}
    \right).
  \end{align}
  Provided that \eqref{eq:const-small} holds, this directly implies
  \eqref{eq:heat-estimate-variable}.
  \par
  Using now estimate \eqref{eq:heat-estimate-variable} in eq. \eqref{eq:heat-variable},
  we get
  \begin{equation}
    \label{eq:dt-Lp-estimate}
    \|\partial_t v\|_{L^p(\Omega_T)}
    \leq
    (1 + b \,D_{a,b,p})\, \|f\|_{L^p(\Omega_T)}.
  \end{equation}
  Taking into account that $v(T,x) = 0$ for $x \in \Omega$,
  \begin{multline}\label{muli}
    \|v(0,\cdot)\|_{L^p(\Omega)}
    =
    \Big\| \int_0^T \partial_t v(t,\cdot) \,dt \Big\|_{L^p(\Omega)}
    \\
    \leq \int_0^T \| \partial_t v(t,\cdot)\|_{L^p(\Omega)} \,dt
    \leq \|\partial_t v(t,\cdot)\|_{L^p(\Omega_T)} T^{1/p'},    
  \end{multline}
  using H\"older's inequality in the last step. Together with
  \eqref{eq:dt-Lp-estimate}, this proves Lemma~\ref{lem23}.
\end{proof}

From the previous Lemmas, we obtain by duality Proposition \ref{propun}.

\begin{proof}[Proof of Proposition \ref{propun}]
Take any $f \in L^{p'}(\Omega_T)$, and consider $v$ the solution of the backward heat equation \eqref{eq:heat-variable}. Testing \eqref{eq:heat-variable} with the solution $u$ of \eqref{eq:heat-variable-forward}, one easily checks that
  \begin{equation*}
    \frac{d}{dt} \left(\, \int_\Omega u(t,x)\, v(t,x) \,dx \right)
    = \int_\Omega u(t,x)\, f(t,x) \,dx,
  \end{equation*}
  which implies that
  \begin{multline*}
    \bigg| \int_{\Omega_T} u f \bigg| = \bigg| - \int_\Omega u_0(x) v(0,x)\, dx \bigg|
    \leq
    \|u_0\|_{L^p(\Omega)} \|v(0,\cdot)\|_{L^{p'}(\Omega)}
    \\
    \leq
    (1 + b \,D_{a,b,p'})\, T^{1/p} \,
    \|u_0\|_{L^p(\Omega)} \,
     \|f\|_{L^{p'}(\Omega_T)},
  \end{multline*}
  where we have used \eqref{eq:v0-bound} for the last inequality, with $p$ replaced by $p'$.
 As this holds for an arbitrary $f \in L^{p'}(\Omega_T)$, we conclude
  that \eqref{eq:heat-estimate-forward} holds, and Proposition~\ref{propun} is proven.
\end{proof}

\begin{rem} \label{rein}
In fact, one can observe that estimate (\ref{muli}) is not optimal. 
 One can show using the properties of the heat equation that
$$\|v(0,\cdot)\|_{L^r(\Omega)} \le C_T\, \left( \|\pa_t v\|_{L^p(\Omega_T)} +
\|\Delta_x v\|_{L^p(\Omega_T)} \right), $$
for any $r < \frac{p\,N}{N + 2 - 2p}$, and any $T>0$ [$r$ can be taken as large as wanted if $2p > N +2$]. 
\medskip

As a consequence, in the proof by duality of Proposition \ref{propun},
the norm $\|u_0\|_{L^p(\Omega)}$ can be replaced by the weaker norm
$\|u_0\|_{L^q(\Omega)}$, for any $q>p/(1 + 2/N)$. This improvement
allows one to consider more singular initial data in the
reaction-diffusion problems studied in the sequel.
\end{rem}

\section{The ``four species'' equation}
\label{sectrois}

We now turn to the application of Proposition \ref{propun} to the
``four species'' system~(\ref{trois}).

\subsection{A general \emph{a priori} estimate}
\label{sec:4species-estimates}

We begin with the following {\sl{a priori}} estimate for the ``four species'' equation,
 which is a direct consequence of Proposition \ref{propun}:

\begin{lem}
  \label{lem31}
  Let $\Omega$ be a bounded domain of $\R^N$ with smooth
  ($C^{2+\alpha}$, $\alpha>0$) boundary $\partial\Omega$ and $T > 0$.
  Consider a weak solution $\{a_i\}_{i=1,\dots,4}$ to system
  (\ref{trois}) with homogeneous Neumann boundary conditions
  (\ref{CN}) on $[0,T]$ and initial condition $a_{i0 }\in L^p(\Omega)$
  ($i=1..4$) for some $p > 2$, and diffusion rates $d_i>0$ ($i=1..4$).
  We denote
  \begin{equation}
    \label{eq:def-a-b}
    a := \min_{i=1,\dots,4} \{d_i\},
    \quad
    b := \max_{i=1,\dots,4} \{d_i\},
  \end{equation}
 and assume that 
  \begin{equation}
    \label{eq:const-small-q-2}
    C_{\frac{a+b}2,p'}\, \frac{b-a}{2} < 1 .
  \end{equation}
   \par
Then
  \begin{equation}
    \label{eq:heat-estimate-forward-2}
    \|a_i\|_{L^p(\Omega_T)}
    \leq
    (1 + b \,D_{a,b,p'})\, T^{1/p} \,
    \bigg\|\sum_{j=1}^4 a_{j0}\bigg\|_{L^p(\Omega)},
  \end{equation}
  for  $i = 1..4$.
\par
Here, $C_{m,p}$ and $D_{a,b,q}$ are the constants defined in Proposition \ref{propun}.
\end{lem}

\begin{proof}[Proof of Lemma \ref{lem31}]
  We call $u := \sum_{i=1}^4 a_i$ the total local mass of the
  system. Then $u$ satisfies the forward heat equation
  (\ref{eq:heat-variable-forward}) with
  \begin{equation}
    \label{eq:def-M}
    M(t,x) := \frac{\sum_{i=1}^4 d_i\, a_i(t,x)}{\sum_{i=1}^4 a_i(t,x)}\, .
  \end{equation}
  We observe that the bound (\ref{eq:M-bounds}) holds. Moreover,
  assumption \eqref{eq:const-small-q-2} is identical to
  assumption~\eqref{eq:const-small-q}.  As a consequence, thanks to
  Proposition \ref{propun}, we end up with
  estimate~(\ref{eq:heat-estimate-forward-2}).
\end{proof}

Next we turn to a Lemma which is specially devised for the
two-dimensional case. Note that it does not depend on the size of
$b-a$: \bigskip

\begin{lem}
  \label{lem32}
  Let $\Omega$ be a bounded domain of $\R^2$ with smooth
  ($C^{2+\alpha}$) boundary $\partial\Omega$, $T>0$ and diffusion
  rates $d_i>0$ ($i=1..4$). We still use the notation
  (\ref{eq:def-a-b}).  \smallskip
 
  Then, one can find two constants $K_1>0$ and $p>2$ depending on
  $a,b$ and $\Omega$, such that any weak solution $(a_i)_{i=1..4}$ of
  system (\ref{trois}) with the homogenenous Neumann boundary
  conditions (\ref{CN}) and initial conditions
  in $L^{p}(\Omega)$ satisfies
  \begin{equation}\label{ne}
    \|a_i\|_{L^{p}(\Omega_T)} \le K_1\, (1 + T)^{1/2} \,   
    \bigg\|\sum_{i=1}^4 a_{i0}\bigg\|_{L^p(\Omega)}. 
  \end{equation}
\end{lem}

\begin{proof}[Proof of Lemma \ref{lem32}]
At first we shall deduce that for any $m>0$ and for  $3/2\le r \le 2$ the following esimtate holds:
\begin{equation}\label{neuf}
C_{m,r} \le m^{- \frac{4}{r}\left(r-\frac{3}{2}\right)}\, (C_{m,3/2})^{\frac{3}{r}\left(2-r\right)} .
\end{equation}
Indeed, multiplying \eqref{quinze} by $\Delta_x v$ and integrating on $\Omega_T$, we easily obtain
  \begin{equation}
    \label{eq:heat-estimate-p=2}
    \|\Delta_x v\|_{L^2(\Omega_T)}
    \leq
    \frac{1}{m} \, \|f\|_{L^2(\Omega_T)},
  \end{equation}
so that  $C_{m,2} \le m^{-1}$ for $m > 0$. 

Then, an interpolation with \eqref{eq:heat-estimate-p} for
$p=\frac{3}{2}$ gives for all $r \in [3/2,2]$
  \begin{equation}
    \label{eq:heat-estimate-interpolation}
    \|\Delta_x v\|_{L^r(\Omega_T)}
    \leq m^{-\theta}\,  (C_{m,3/2})^{1-\theta} \, \|f\|_{L^r(\Omega_T)},
  \end{equation}
  with the interpolation exponent $\theta \in [0,1]$ satisfying
$$  
\frac{\theta}{2} + \frac{1-\theta}{3/2} = \frac{1}{r}
\quad\Rightarrow\quad \theta = \frac{4r-6}{r},
$$
which yields (\ref{neuf}).
\par
Using now (\ref{neuf}) for $r=p'$ and $m = \frac{a+b}2$, we see that
eq. (\ref{eq:const-small-q-2}) is satisfied as 
soon as the following inequality is satisfied:
$$
-\theta(p')\ln(m)+(1-\theta(p'))\ln(C_{m,3/2})+\ln\left(\frac{b-a}{2}\right)<0,
$$
which yields with $1-\theta(p')=\frac{3}{p'}(2-p')$ 
\begin{equation}\label{dfg}
\frac{3}{p'}\, (2-p')\,\left( \log(C_{\frac{a+b}2, 3/2}) +  \log \left(\frac{a+b}2 \right)\right) < \log\left(\frac{a+b}{b-a}\right). 
\end{equation}
Thus, the condition (\ref{eq:const-small-q-2}) is  satisfied provided that $\frac{a+b}2\, C_{\frac{a+b}2, 3/2} >1$ as soon as
$$ 2 -p' < \frac{p'}3\, \frac{\log \left(\frac{a+b}{b-a}\right) }{\log( \frac{a+b}2\,
C_{\frac{a+b}2, 3/2} )} , $$
and, therefore, as soon as we choose a $p' \in [3/2,2]$ satisfying
$$ 2-p' < \frac12\, \frac{\log \left(\frac{a+b}{b-a}\right) }{\log(\frac{a+b}2\,
C_{\frac{a+b}2, 3/2} ) }. $$
Note that in the case $\frac{a+b}2\, C_{\frac{a+b}2, 3/2} <1$, the eq. \eqref{dfg} implies that condition (\ref{eq:const-small-q-2}) is always satisfied for all $p' \in [3/2,2]$ since $\log(\frac{a+b}{b-a})>0$. 
\end{proof}

\subsection{Polynomial w.r.t. time bootstrap estimates}

We prove below a standard estimate for the heat equation (with Neumann
boundary condition) which amounts to proving that the corresponding
Green function has the same singularity as the Green function in the
case of the whole $x$-space $\R^N$
\par
We put the stress on the dependence of the constants w.r.t. the length
$T$ of the time interval (this dependence is only tracked with very
great effort in the classical books like \cite{LSU}, cf. Remark
\ref{rema} below).
 
\begin{lem}
  \label{lem33}
  Let $\Omega$ be a bounded domain of $\R^N$ with smooth
  ($C^{2+\alpha}$) boundary $\partial\Omega$ and $T > 0$.  Let $u$ be
  a solution of the forward heat equation with r.h.s. $f\in
  L^q(\Omega_T)$ and homogeneous Neumann boundary conditions.
\begin{equation}
    \label{eq:heat-Neumann}
    \left\{
      \begin{aligned}
        &\partial_t u - d \,\Delta_x u = f &&\quad \text{ on } \quad \Omega_T,
        \\
        &u(0,x) = u_0(x) &&\quad \text{ for } \quad x \in \Omega,
        \\
        &\nabla u \cdot \nu(x) = 0 &&\quad \text{ on } \quad [0,T]
        \times \partial \Omega.
      \end{aligned}
    \right.
  \end{equation}
\begin{description}
\item If $1<q < \frac{N+2}{2}$, we consider $s =
  \frac{q\,(N+2)}{N+2-2q}>0,$ and assume further that the initial
  datum $u_0$ belong to $L^s(\Omega)$.
  \par
  Then, for any $0 < \epsilon < s-1$,
  there is a constant $C_T > 0$ depending on $N$, $\Omega$, $\epsilon$, $d$, $q$,  $\|u_0\|_{L^s(\Omega)}$, $\|f\|_{L^q(\Omega_T)}$ and which has an at most polynomial dependence w.r.t. $T$, such that
  \begin{equation}
    \label{eq:heat-Lp-gain}
    \|u\|_{L^{s-\epsilon}(\Omega_T)} \leq C_T.
  \end{equation}
\item  On the other hand, if $q \ge \frac{N+2}{2}$,  we assume that the initial datum $u_0$ belongs to $L^{\infty}(\Omega)$.
  Then for any $r\in [1, +\infty[$, there exists a constant $C_T$ depending on $N$, $\Omega$, $r$, $d$, $q$, 
  $\|u_0\|_{L^{\infty}(\Omega)}$, $\|f\|_{L^q(\Omega_T)}$ and which has an at most polynomial dependence w.r.t. $T$, such that
  \begin{equation}
    \label{eq:heat-Linfty-gain}
    \|u\|_{L^{r}(\Omega_T)} \leq C_T.
  \end{equation}
\end{description}
\end{lem}

\begin{rem} \label{rema} The statement of the above lemma is classical
  (and non even optimal) except that it crucially shows that the
  regularising effect of a parabolic equation involves constants which
  depend polynomially on the time interval $[0,T]$ for all $T>0$. In
  1D, this was already shown in \cite{DF08} using a Fourier
  representation of the solution. For general domains however, the
  polynomial dependence of the constants seems to be nowhere in the
  literature. Moreover, tracking the constants, for instance, in the
  approach of \cite{LSU} (where the Green function for the half-space
  problem are used locally along the sufficiently smooth boundary as
  transformed approximating problem) is much more difficult than the
  following proof.
\end{rem}

\begin{proof}[Proof of Lemma \ref{lem33}]
We consider a solution of eq. \eqref{eq:heat-Neumann}.
\smallskip

Step 1: Setting $p_0=q$ and testing \eqref{eq:heat-Neumann} with 
$p_0\,|u|^{p_0-1}\mathrm{sgn}(u)$ (more precisely by testing with a smoothed version of the modulus $|u|$ and its derivative $\mathrm{sgn}(u)$ and letting then the smoothing 
tend to zero) we obtain by integration-by-parts and H\"olders inequality with the constant $C_0(p_0):=\frac{4(p_0-1)\,d}{p_0}$~:
\begin{equation}\label{eq:heat-lp}
 \frac{d}{dt} \bigg( \|u\|_{L^{p_0}_x}^{p_0} \bigg) 
 +  C_0(p_0,d)\iO \left|\nabla_x (u^{p_0/2})\right|^2dx \le p_0\,\|f\|_{L^{p_0}_x}\,(\|u\|_{L^{p_0}_x}^{p_0})^{\frac{p_0-1}{p_0}}.
\end{equation}
Then, the Gronwall lemma 
$
\dot y \le \alpha(t) y^{1-\frac{1}{p_0}} \Rightarrow\
y(T) \le \left[y(0)^{\frac{1}{p_0}} + \frac{1}{p_0}\int_0^T \alpha(t)\,dt\right]^{p_0} 
$
yields for all $T>0$~:
\begin{align*}
\|u\|_{L^{p_0}_x}^{p_0}(T)& \le \left[\|u_0\|_{L^{p_0}_x}+\int_0^T \|f\|_{L^{p_0}_x}(s)\,ds\right]^{p_0} \le 
\left[\|u_0\|_{L^{p_0}_x}+\|f\|_{L^{p_0}_{t,x}}T^{\frac{p_0-1}{p_0}}\right]^{p_0}\nonumber\\
&\le 2^{p_0-1}\left[\|u_0\|_{L^{p_0}_x}^{p_0} + \|f\|_{L^{p_0}_{t,x}}^{p_0} T^{p_0-1}\right]:=C_{T,0}.\nonumber
\end{align*}
We thus conclude that there is a constant $C_{T,0}$ depending only on 
$p_0$, $d$, $\|u_0\|_{L^{p_0}_x}$, $\|f\|_{L^{p_0}_{t,x}}$ and polynomially on $T$ such that
\begin{equation}
\sup\limits_{t\in[0,T]}\|u\|_{L^{p_0}_x}^{p_0}(t) \le C_{T,0}. \label{eq:CT0}
\end{equation}

Step 2: Gradient estimate and Sobolev embedding. 
The integration-in-time of \eqref{eq:heat-lp} and H\"older's inequality show
\begin{equation*}
C_0\iOT \left|\nabla_x (u^{p_0/2})\right|^2\,dx \le 
\|u_0\|_{L^{p_0}_x}^{p_0}+ p_0\, \|f\|_{L^{p_0}_{t,x}}\, \|u\|_{L^{p_0}_{t,x}}^{p_0-1}.
 \end{equation*}
The above estimate and Sobolev's embedding for $H^1$ with constant $C_{\mathrm{S}}$ yields for $s_0<\infty$ for $N=2$ and $s_0=\frac{p_0N}{N-2}$ for $N>2$
(together with Young's inequality)
\begin{align}
 \int_0^T \|u\|_{L^{s_0}_x}^{p_0}\,dt &\le \frac{C^2_{\mathrm{S}}}{C_0}\left[\|u_0\|_{L^{p_0}_x}^{p_0} + p_0\, \|f\|_{L^{p_0}_{t,x}} (T\,C_{T,0})^{\frac{p_0-1}{p_0}}
\right]
:= D_{T,0}, \label{eq:DT0}
 \end{align}
where $D_{T,0}$ is a constant depending only on 
$p_0$, $d$, $\|u_0\|_{L^{p_0}_x}$, $\|f\|_{L^{p_0}_{t,x}}$ and polynomially on $T$. 
\medskip

In the Steps 3 and 4, we construct a sequence of exponents $p_n$, $s_n$ and bounds
\begin{align}
\sup\limits_{t\in[0,T]}\|u\|_{L^{p_n}_x}^{p_n}(t) &\le C_{T,n}, \label{eq:CTn}\\
 \int_0^T \|u\|_{L^{s_n}_x}^{p_n}\,dt &\le D_{T,n}. \label{eq:DTn}
\end{align}
In particular we set $s_n<\infty$ if $N=2$ and $s_n=p_n \frac{N}{N-2}$ if $N\ge3$. 
\medskip

Step 3: Iteration of \eqref{eq:CTn}: Similar to Step 1 we test \eqref{eq:heat-Neumann} with 
$p_{n+1}\,u^{p_{n+1}-1}$~:
\begin{equation}
\frac{d}{dt} \bigg(  \|u\|_{L^{p_{n+1}}_x}^{p_{n+1}} \bigg)
 + C_{n+1} \iO \left|\nabla_x (u^{p_{n+1}/2})\right|^2dx = p_{n+1}\,\iO f\, u^{p_{n+1}-1}\,dx, 
\label{eq:lpn}
\end{equation}
where $C_{n+1}(p_{n+1}):=\frac{4(p_{n+1}-1)\,d}{p_{n+1}}$. In order to iterate the bound \eqref{eq:CTn},
we fix the exponent $p_{n+1}$ by introducing the $n$-independent exponent 
\begin{equation}
r = \frac{\frac{s_n}{p_0}-1}{s_n-p_{n+1}}:=1-\frac{2}{N}+\frac{2}{N p_0}
\quad\text{iff}\quad N\ge3\quad\text{which satisfy}\ \frac{1}{p_0}<r<1, 
\label{eq:r}
\end{equation}
and any $r$ satisfying $\frac{1}{p_0}<r<1$ if $N=2$. Then,
we estimate with $p_{n+1}-1=p_{n+1}(1-r)+s_n(r-1/p_0)$ the above right-hand side of \eqref{eq:lpn} by H\"older's inequality
\begin{align}
\iO f\,u^{s_n(r-1/p_0)}\, u^{p_{n+1}(1-r)} \,dx \le \|f\|_{L^{p_0}_x}\, \|u\|_{L^{s_{n}}_x}^{s_n(r-1/p_0)}\, \|u\|_{L^{p_{n+1}}_x}^{p_{n+1}(1-r)},
\label{eq:lpnHolder}
\end{align}
and a Gronwall estimate for $\dot y \le \alpha(t) y^{1-r}$
yields 
\begin{multline}
\|u\|_{L^{p_{n+1}}_x}^{p_{n+1}}(T)\le \left[\|u_0\|_{L^{p_{n+1}}_x}^{r\,p_{n+1}}+p_{n+1}r \int_0^T \|f\|_{L^{p_0}_x}\, \|u\|_{L^{s_{n}}_x}^{s_n(r-1/p_0)} \,dt\right]^{1/r}\nonumber\\
\le \left[\|u_0\|_{L^{p_{n+1}}_x}^{r\,p_{n+1}}+p_{n+1}r \|f\|_{L^{p_0}_{t,x}} \biggl(\int_0^T  \|u\|_{L^{s_{n}}_x}^{s_n(r-1/p_0)\frac{p_0}{p_0-1}} \,dt\biggr)^{\frac{p_0-1}{p_0}}\right]^{1/r}.
\end{multline}
Thus, by the definition of $r$ we have   
$
s_n(r-1/p_0)\frac{p_0}{p_0-1} = s_n\frac{N-2}{N}=p_n,
$
and we are able to use the bound \eqref{eq:DTn} to obtain
\begin{equation}
\|u\|_{L^{p_{n+1}}_x}^{p_{n+1}}(T)\le \left[\|u_0\|_{L^{p_{n+1}}_x}^{r\,p_{n+1}}+p_{n+1}r \|f\|_{L^{p_0}_{t,x}} D_{T,n}^{\frac{p_0-1}{p_0}}\right]^{1/r} =: C_{T,n+1}.
\label{eq:heat-lpn}
\end{equation}

Step 4: Iteration of \eqref{eq:DTn}: Returning to \eqref{eq:lpn} and \eqref{eq:lpnHolder}, we collect
\begin{multline*}
\frac{1}{r}\, \frac{d}{dt} \bigg( [\|u\|_{L^{p_{n+1}}_x}^{p_{n+1}}]^r \bigg)
 +  C_{n+1}\,\|u\|_{L^{p_{n+1}}_x}^{(r-1)\,p_{n+1}}\iO \left|\nabla_x (u^{\frac{p_{n+1}}{2}})\right|^2dx\\ \le p_{n+1}\,\|f\|_{L^{p_0}_x}\, \|u\|_{L^{s_{n}}_x}^{s_n(r-\frac{1}{p_0})}.
\end{multline*}
Since $r<1$ we have $\|u\|_{L^{p_{n+1}}_x}^{p_{n+1}(r-1)}\ge (C_{T,n+1})^{r-1}$ and integration-in-time and H\"older's inequality as in \eqref{eq:lpnHolder} yield
\begin{equation*}
(C_{T,n+1})^{r-1}\, r\, C_{n+1}\iOT\left|\nabla_x (u^{p_{n+1}/2})\right|^2dx \le \|u_0\|_{L^{p_{n+1}}_x}^{r\,p_{n+1}} +  
r\,p_{n+1}\,\|f\|_{L^{p_0}_{t,x}}\, D_{T,n}^{\frac{p_0-1}{p_0}}.
\end{equation*}
Finally, with $s_{n+1}=p_{n+1}\frac{N}{N-2}$ if $N>2$ and using Sobolev's embedding,  
\begin{equation*}
\int_0^T \|u\|_{L^{s_{n+1}}_x}^{p_{n+1}}\,dt \le \frac{C^2_{\mathrm{S}} C_{T,n+1}^{1-r}}{C_{n+1}\,r}\left[\|u_0\|_{L^{p_{n+1}}_x}^{r\,p_{n+1}} +  
r\,p_{n+1}\|f\|_{L^{p_0}_{t,x}}\, D_{T,n}^{\frac{p_0-1}{p_0}}\right] =: D_{T,n+1}.
\end{equation*}

Step 5: Iteration in $n$. From the definition of $r$ in \eqref{eq:r} it follows that 
$p_{n+1} = s_n (1-\frac{1}{rp_0})+\frac{1}{r}$
and thus $p_{n+1}<\infty$ if $N=2$ and for dimensions 
$N\ge3$, where $s_n=p_n\frac{N}{N-2}$~: 
$$
p_{n+1} = p_n \frac{N}{N-2}\left(1-\frac{1}{rp_0}\right)+\frac{1}{r} = p_n \frac{N(p_0-1)}{p_0(N-2)+2}+\frac{N\,p_0}{p_0(N-2)+2}
$$
which has the fixed point $p_{\infty} = \frac{N p_0}{N+2-2p_0}$ and 
$$
p_{\infty}<0 \quad\Longleftrightarrow\quad p_0>\frac{N+2}{2} 
\quad\Longleftrightarrow\quad \frac{N(p_0-1)}{p_0(N-2)+2}>1.
$$
Thus, with $p_0=q$ and $f\in L^q_{t,x}$ we distinguish the cases 
\begin{equation}
\begin{cases}
q<\frac{N+2}{2} \quad\text{where}\quad p_n \xrightarrow{n\to\infty} p_{\infty}>q\quad\text{for}\quad q>1,\\
q \ge \frac{N+2}{2} \quad\text{where}\quad p_n \xrightarrow{n\to\infty} +\infty.
\end{cases}
\end{equation}
In dimension $N=2$ we can always choose $p_{\infty}<+\infty$ to be arbitrarily large. Note that for any $n$ in the iteration, the constants $C_{T,n}$ and $D_{T,n}$ are polynomial with respect to $T$!
\medskip

Step 6: Interpolation of \eqref{eq:CTn} and \eqref{eq:DTn}
in the cases $p_{\infty}<+\infty$ (and thus $N\ge3$). 
For any $n$ we use H\"older's inequality
\begin{align*}
\iOT u^{p_n}\,u^{p_n\frac{2}{N}}\,dxdt\le \int_0^T \|u\|_{L^{s_n}_x}^{p_n}\,\|u\|_{L^{p_n}_x}^{\frac{2}{N} p_n}\,dt\le C_{T,n}^{\frac{2}{N}} D_{T,n}.
\end{align*}
In the limit $n\to\infty$ we find $p_{\infty}\frac{N+2}{N}=\frac{(N+2)q}{N+2-2q}>0$.

Thus, for all $\eps>0$ and in all dimensions $N\ge3$ we obtain after finitely many iterations the following bound   
$$
\| u \|_{L^{p_{\infty}\frac{N+2}{N}-\eps}_{t,x}} \le C_{T}(\|u_0\|_{L^{p_{\infty}}_x}, \|f\|_{L^{q}_{t,x}}, q,d, C_S),
$$ 
where $C_T$ is a constant depending only on  $\|u_0\|_{L^{p_{\infty}}_x}$, $\|f\|_{L^{q}_{t,x}}$, $d$, $q$, the Sobolev constant $C_{\mathrm{S}}$, and $T$. Moreover, $C_T$ depends polynomially on $T$.
\end{proof}

\begin{rem}
We remark that
$$
\frac{1}{p_{\infty}\frac{N+2}{N}-\eps} = \frac{1}{q}-1+\frac{N}{N+2}+O(\eps),
$$
which corresponds to the regularity expected by convolution with the heat kernel being in $L^{\frac{N+2}{N}-\mu}$, for all $\mu>0$.
\end{rem}

When applied to the quadratic ``four species'' eq. (\ref{trois}), the bootstrap above yields the following
lemma:

\begin{lem}
%[Bootstrap argument]
  \label{lem36}
  Let $\Omega$ be a bounded smooth ($C^{2+\alpha}$) open subset of $\R^N$ and $T > 0$.
  Consider then a weak solution $\{a_i\}_{i=1..4}$ to equation
  (\ref{trois}), (\ref{CN}) on $[0,T]$ with initial condition $\{a_{i0}\}_{i=1..4}\in
  L^{\infty}(\Omega_T)$
 and diffusion rates $d_i>0$, $[i=1..4]$.
\par
Assume that $\{a_i\}_{i=1..4}$ 
  lie in $L^{q_0}(\Omega_T)$ for some $q_0 > (N+2)/2$, and that $\|a_i\|_{L^{q_0}(\Omega_T)}$ grows at most polynomially w.r.t. $T$ for $i=1..4$.
\par
 Then, for any $r \in [1, +\infty[$, we have 
$\{a_i\}_{i=1..4} \in L^{r}(\Omega_T)$ and  $\|a_i\|_{L^{r}(\Omega_T)}$ grows at most polynomially w.r.t. $T$.
\end{lem}

\begin{proof}[Proof of Lemma \ref{lem36}] 
  We use Lemma \ref{lem33} repeatedly. In general, if $a_i
  \in L^q(\Omega_T)$ for some $q > 2$, then $a_i a_j \in
  L^{q/2}(\Omega_T)$ ($i,j=1,\dots,4$). Hence the right-hand side of
  eq. \eqref{trois} is in $L^{q/2}$, so from Lemma
  \ref{lem33} we have
  \begin{equation*}
    a_i \in L^{r- \delta}(\Omega_T)
    \quad \text{ with }
    r =
    \begin{cases}
      \frac{1}{2} \frac{q(N+2)}{N+2-q}
      \quad &\text{ if } 1<q < N+2,
      \\
      \infty \quad &\text{ if } q \ge N+2,
    \end{cases}
  \end{equation*}
for any $\delta>0$.
If we define the sequence $q_n$ starting with the $q_0$ given in the Lemma, and satisfying
  \begin{equation}
    \label{eq:iteration}
    q_{n+1} = \frac{1}{2} \frac{q_n(N+2)}{N+2-q_n},
 \quad\text{as long as}\quad q_n < N+2,
   \end{equation}
one can readily check that $q_{n+1}>q_n $ is equivalent to 
$q_n>\frac{N+2}{2}$ and thus $\frac{q_{n+2}}{q_{n+1}}>\frac{q_{n+1}}{q_{n}}$ and we obtain within finitely many iterations that    
  \begin{equation*}
    a_i \in L^{q_n - \delta}(\Omega_T), \quad\text{with}\quad q_n >
  N+2     \quad \text{for some } n \geq 0, \quad \text{and any} \quad \delta>0.
  \end{equation*}
Thus, applying once more Lemma \ref{lem33}, we 
end up with $a_i \in L^{r}(\Omega_T)$ for any $r \in [1, +\infty[$.
\end{proof}

In order to get an $L^{\infty}$ estimate, we need one more computation:

\begin{lem}
%[Higher regularity bootstrap]
  \label{lem37}
 Let $\Omega$ be a bounded domain of $\R^N$ with smooth ($C^{2+\alpha}$) boundary $\partial\Omega$ and $T > 0$.  Consider a solution $\{a_i\}_{i=1..4}$ to the system
  \eqref{trois}, \eqref{CN} on $[0,T]$ with initial condition $\{a_{i0 }\}_{i=1..4}\in
  L^{\infty}(\Omega)$ and diffusion rates $d_i>0$ $[i=1..4]$.
\par
Assume that $\{a_i\}_{i=1..4}$ is a weak solution to the system \eqref{trois}, \eqref{CN} on $[0,T]$
satisfying $a_i \in L^{q_0}(\Omega_T)$ for $i=1..4$ and some $q_0 > (N+2)/2$. Also
assume that $\|a_i\|_{L^{q_0}(\Omega_T)}$ grows at most polynomially w.r.t. $T$.
\par
 Then,  
$$
\|a_i\|_{L^{\infty}(\Omega_T)}\le C_T, \qquad i=1..4,
$$ 
where $C_T$ grows at most polynomially w.r.t. $T$.
\end{lem}

\begin{proof}[Proof of Lemma \ref{lem37}]
Using
 Lemma \ref{lem36}, we know that
$ \pa_t a_i - d_i\, \Delta_x a_i $ lies in $L^r(\Omega_T)$ for all $r \in [1, +\infty[$, with a
norm in $L^r(\Omega_T)$ which grows at most polynomially w.r.t. $T$.
\par
Using the result of \cite{JFA}, we obtain that the derivatives $\pa_t a_i$ and $\pa_{x_j x_k} a_i$ lie in 
$L^r(\Omega_T)$ for all $r \in [1, +\infty[$, with a
norm in $L^r(\Omega_T)$ which grows at most polynomially w.r.t. $T$.
\par
A standard Sobolev inequality (in $\R^{N+1}$) together with the use of an extension/restriction operator
implies that $a_i$ lies in $L^{\infty}(\Omega_T)$, with a
norm in $L^{\infty}(\Omega_T)$ which grows at most polynomially w.r.t. $T$.
\end{proof}

\subsection{Existence of bounded solutions and large time behaviour for the ``four species'' model}
\label{sec:4species-existence}

 The results of the previous subsections enable us to prove Proposition~\ref{propdeux} and
 Proposition~\ref{proptrois}.
\bigskip

\begin{proof}[Proof of Proposition \ref{propdeux}] The existence of a
  weak solution to eq. (\ref{trois}), (\ref{CN}) [with given initial
  data in $L^{\infty}(\Omega)$] is already known (cf. \cite{DFPV}). We
  observe that thanks to Lemma \ref{lem32}, we have a $L^p(\Omega_T)$
  estimate for $a_i$ ($i=1..4$) for some $p>2$. According to Lemma
  \ref{lem37}, the $a_i$ ($i=1..4$) lie in fact in
  $L^{\infty}(\Omega_T)$.  Since $T$ can be taken arbitrarily large,
  we end up with solutions defined on all $\R_+ \times \Omega$.
  \par
  Moreover, still according to Lemma \ref{lem37}, the
  $L^{\infty}(\Omega_T)$ bounds of the $a_i$ ($i=1..4$) are at most
  polynomially increasing w.r.t. $T$.  Using the
  entropy/entropy-dissipation estimate proved in \cite{DFEqua} (or
  \cite{DF08}, which used the assumption $N=1$ only in order to show
  at most polynomially growing $L^{\infty}$ bounds), we end up with
  the exponential decay towards equilibrium (\ref{quatre}), and in
  particular we get a uniform-in-time bound for the $a_i$ in
  $L^{\infty}(\R_+ \times \Omega)$.
\end{proof}
\medskip

\begin{proof}[Proof of Proposition \ref{proptrois}] Once again, the
  existence of a weak solution to eq. (\ref{trois}), (\ref{CN}) [with
  given initial data in $L^{\infty}(\Omega)$], is already known
  (cf. \cite{DFPV}).  Under the smallness assumption made on
  $\delta=b-a$, % We take for example $\delta_0 := [ C_{\frac{a+b}2, (N/2 +1)'} + 1]^{-1}$.
  Lemma~\ref{lem31} implies an $L^p(\Omega_T)$ estimate for $a_i$
  ($i=1..4$) for some $p>N/2 + 1$. Then, Lemma \ref{lem37} ensures
  that the $a_i$ ($i=1..4$) lie in fact in $L^{\infty}(\Omega_T)$.
  The end of the proof (that is, the estimates about the convergence
  towards equilibrium) is exactly the same as in the Proof of
  Proposition \ref{propdeux}.
\end{proof}

\section{Extensions: General chemical kinetics and degenerate diffusion rates}
\label{secquatre}

\begin{proof}[Proof of Proposition \ref{propquatre}] We introduce the approximated system of equations
constituted of the approximated equation (for any $r \in \N^*$)
\begin{equation}\label{genpolappr}
\pa_t a_i^r - d_i\,\Delta_x a_i^r = \frac{(\beta_i-\alpha_i)\left(l \,
\prod\limits_{j=1}^n (a_j^r)^{\alpha_j} - k\, \prod\limits_{j=1}^n (a_j^r)^{\beta_j} \, \right)}{1 + \frac1r\, \left( \sum_{j=1}^n (a_j^r)^2 \right)^{Q/2} } , \qquad i=1..n,
\end{equation}
where $Q= \sup\{ \sum_{i=1}^n \alpha_i, \sum_{i=1}^n \beta_i\}$ is defined as in the statement of the Proposition \ref{propquatre} and assumed to be superquadratic, i.e. we consider $Q\ge3$.
Here, eq. \eqref{genpolappr} is to be considered
together with the homogeneous Neumann boundary condition (\ref{CN}) and a set of smooth approximated 
initial data $\{a_{i0}^r\}_{i=1,..,n}$ (converging a.e. towards $\{a_{i0}\}_{i=1,..,n}$ as 
$r\to \infty$ and bounded in $L^{\infty}(\Omega)$).
\medskip

The existence of a smooth (strong) solution to this approximated system follows from standard existence results of systems of reaction-diffusion equations with bounded and Lipschitz-continuous r.h.s. (cf. \cite{desvillettes_milan,QS}, for example).
\medskip

With the assumption that at least two coefficients $\alpha_i - \beta_i$ are different from zero
and have opposite signs, one can find coefficients $\gamma_i>0$ such that
$$ 
\sum_{i=1}^n \gamma_i\, (\alpha_i - \beta_i) = 0 . $$

At first, we observe then that
$$ 
\pa_t \Bigl(\sum_{i=1}^n \gamma_i\, a_i^r \Bigr) - \Delta_x (M_r\, a_i^r) = 0, 
$$
where  
$$ 
M_r(t,x) = \frac{\sum_{i=1}^n \gamma_i\, d_i\, a_i^r(t,x)}{\sum_{i=1}^n \gamma_i\, a_i^r(t,x)} 
$$
is bounded $a\le M_r(t,x)\le b$ a.e. in $\Omega_T$.
\medskip
  
We now assume that
\begin{equation} \label{sdf}
b-a < 2\,(C_{\frac{a+b}2, Q'})^{-1}, \qquad \frac{1}{Q}+\frac{1}{Q'}=1.
\end{equation}
Using Proposition \ref{propun}, we see that for any $i=1,..,n$, the sequence
$(a_i^r)_{r\in \N}$ is bounded in $L^Q(\Omega_T)$ (for all $T>0$). Moreover, since \eqref{sdf} is a strict inequality, an interpolation argument similar to the one used in Lemma \ref{lem32} implies that $(a_i^r)_{r\in \N}$ is bounded in $L^{Q + \var}(\Omega_T)$ for
some sufficiently small $\var>0$. As a consequence, the quantities $\pa_t a_i^r - d_i\, \Delta_x a_i^r$ are
bounded in $L^{1 + \var/Q}(\Omega_T)$. 

The sequence $(a_i^r)_{r\in \N}$ converges therefore (up to extraction
of a subsequence) a.e. as well as strongly in $L^Q(\Omega_T)$ towards
a limit $a_i \in L^{Q + \var}(\Omega_T)$.  Finally, one can pass to
the limit $r\to\infty$ without difficulties, which ensures that the
limiting concentrations $a_i$ satisfy the original system in the weak
sense.  \medskip

Secondly, if we suppose 
$$ b-a < 2\,(C_{\frac{a+b}2, \frac{(Q-1)\,(N+2)}{(Q-1)\,(N+2) - 2} })^{-1},
$$
Proposition \ref{propun} (and the same interpolation argument as previously) ensures that the weak solution defined above satisfies the extra 
estimate 
$a_i \in L^{z_0}(\Omega_T)$, for some $z_0 > (1 + N/2)\,(Q-1)$. Then, 
$\pa_t a_i - d_i\, \Delta_x a_i \in L^{z_0/Q}(\Omega_T)$, and thanks to the properties of the
heat kernel summarised in Lemma \ref{lem33}, $a_i \in L^p(\Omega_T)$ for all 
$p< z_1$, with $z_1 = \frac{z_0}{Q - \frac2{N+2}\,z_0}$ (or all $p \in [1, +\infty[$ 
if $z_0 \ge Q\, (1 + N/2)$). 

We see that a simple bootstrap ensures that 
  $a_i \in L^p(\Omega_T)$ for all 
$p< z_k$, with $z_k= \frac{z_{k-1}}{Q - \frac2{N+2}\,z_{k-1}}$ (or all $p \in [1, +\infty[$ 
if $z_{k-1} \ge Q\, (1 + N/2)$). The sequence $z_k$ is increasing up to the point when 
$z_k \ge Q\, (1 + N/2)$, therefore we obtain the estimate $a_i \in L^p(\Omega_T)$ for all
$p \in [1, +\infty[$. We proceed as in Lemma \ref{lem37} to get the final estimate
$a_i \in L^{\infty}(\Omega_T)$.
\end{proof}

\begin{proof}[Proof of Proposition \ref{propcinq}] Existence of weak solutions 
(in $L^2(\Omega_T)$) for the set
of equations considered in this Proposition is shown in \cite{DF07}. 
\medskip

By adding the equations satisfied 
by $a_1$ and $a_2$, we see that
\begin{equation}\label{a1a2}
 \pa_t (a_1 + a_2) - \Delta_x (M\, (a_1+a_2)) =0, 
 \end{equation}
where $M(t,x) \in [\inf\{d_1,d_2\}, \sup\{d_1,d_2\}]$ almost everywhere.
\par
Adding the 
equations satisfied 
by $a_2$ and $a_3$, we also see that
\begin{equation}\label{a2a3}
\pa_t (a_2 + a_3) - \Delta_x (M\, (a_2+a_3)) =0, 
\end{equation}
where $M(t,x) \in [\inf\{d_2,d_3\}, \sup\{d_2,d_3\}]$ almost everywhere.
\par
As a consequence, we see that
thanks to
Proposition \ref{propun} and an interpolation
argument similar to the one used in Lemma \ref{lem32}, 
for some $\delta \in ]0, 2[$ (and any $T>0$),
$a_i \in L^{2+\delta}(\Omega_T)$, when $i=1,2,3$.
Then, $a_1\,a_3 \in L^{1+\delta/2}(\Omega_T)$. Since 
$$ \pa_t a_2 - d_2\, \Delta_x a_2 \le a_1\,a_3, $$
we see using the properties of the heat kernel in 2D (cf. Lemma \ref{lem33}) that  
$a_2 \in L^{\frac{2+\delta}{1 -\delta/2} - 0}(\Omega_T)$ (here and in the sequel,
the notation $L^{p-0}$ means $\cap_{q \in [1,p[} L^q$).
\medskip

Using a duality estimate (see e.g. \cite[Lemma 3.4]{Pie10}, \cite[Lemma 33.3]{QS}) for solutions of (\ref{a1a2}), it follows that 
if $a_2 \in L^q(\Omega_T)$ for any $1<q<\infty$ 
then also $a_1$ lies in $L^q(\Omega_T)$. Thus, 
we deduce from the estimate on $a_2$ that
$a_1$, $a_3 \in L^{\frac{2+\delta}{1 -\delta/2} - 0}(\Omega_T)$.
We build the (finite) increasing sequence $p_n \in ]2, 4[$
 such that $p_0 = 2 + \delta$, and 
$\frac1{p_{n+1}} = \frac2{p_n} - \frac12$. We denote by $N_0$ the last index such that
$p_{N_0} < 4$. 
The properties of the heat kernel in 2D (once again and in all the sequel, cf. Lemma \ref{lem33} for a precise exposition of those properties)  and the duality estimate of \cite{QS}
implies that 
    $a_i \in L^{p_{N_0+1}}(\Omega_T)$, when $i=1,2,3$.
    A last application of the properties of the heat kernel in 2D 
and the duality estimate of \cite{Pie10,QS} shows that
$a_i \in L^{\infty - 0}(\Omega_T)$, when $i=1,2,3$.
Finally, thanks to a computation similar to the one in Lemma~\ref{lem37},
 $a_2 \in L^{\infty}(\Omega_T)$.
\medskip

 Observing that 
  \begin{equation}\label{eqa4}
\pa_t a_4 \le a_1\, a_3, 
\end{equation}
 we also see (performing the integration in time) 
that $a_4 \in L^{\infty - 0}(\Omega_T)$.
 Then, for $i=1,3$,
 $$ \pa_t a_i - d_i\, \Delta_x a_i \le a_2\,a_4, $$
 so that thanks to the same computation as above (similar to the one in Lemma~\ref{lem37}),
 $$ a_1, a_3 \in L^{\infty}(\Omega_T). $$
 Using once again eq. (\ref{eqa4}), we obtain that 
 $a_4 \in L^{\infty}(\Omega_T)$.
\end{proof}

\section*{Acknowledgment}
The authors would like to thank very much Prof. Felix Otto for pointing out to us the Meyer's type estimates, 
which formed the starting point of our work.


\begin{thebibliography}{}
\bibitem[Ama]{Ama} H. Amann, \textit{Global existence for semilinear parabolic problems.} 
J. Reine Angew. Math. \textbf{360}, (1985), pp. 47--83.

\bibitem[CDF10a]{CDF10a} J. A. Ca\~{n}izo, L. Desvillettes, K. Fellner, \textit{Regularity and mass conservation for discrete coagulation-fragmentation equations with diffusion}, Ann. Inst. H. Poincar\'{e} (C) Anal. Non Lin\'{e}aire, \textbf{27} no.2 (2010) pp. 639--654. 

\bibitem[CDF10b]{CDF10b} J. A. Ca\~{n}izo, L. Desvillettes, K. 
Fellner, \textit{Absence of Gelation for Models of 
Coag\-ulation-Fragmentation with Degenerate Diffusion}, Il Nuovo Cimento, Proceedings of the ICTT, \textbf{33} no.1 (2010) pp. 79--86.  

\bibitem[CV]{CV} C. Caputo, A. Vasseur, \textit{Global regularity of solutions to systems of reaction-diffusion
with sub-quadratic growth in any dimension},
 Commun. Partial Differential Equations, \textbf{34} no.10--12 (2009) 
pp. 1228--1250. 

\bibitem[DeG]{DeG} E. De Giorgi, \textit{Sulla differenziabilit\`{a} e l’analiticit\`{a} delle estremaili degli integrali multipli regolari}, Mem. Accad. Sci. Torino. Cl. Sci. Fis. Math. Nat., \textbf{3} (1957), pp.25--43.

\bibitem[D]{desvillettes_milan} L. Desvillettes, \textit{About Entropy Methods for Reaction-Diffusion Equations}, Rivista di Matematica
 dell'Universit\`a di Parma, {\textbf{7}} n.7 (2007) pp. 81--123.

\bibitem[DFEqua]{DFEqua} L. Desvillettes, K. Fellner, \textit{Entropy Methods for Reaction-Diffusion Equations with Degenerate Diffusion 
Arising in Reversible Chemistry}, accepted in the unfortunately likely never to be printed Proceedings of the Equadiff 2007.

\bibitem[DF08]{DF08} L. Desvillettes, K. Fellner, \textit{Entropy Methods for 
Reaction-Diffusion Equations: Slowly Growing A-priori Bounds}, Revista Matem\'atica Iberoamericana,
\textbf{24} no. 2 (2008) pp. 407--431.

\bibitem[DFPV]{DFPV} L. Desvillettes, K. Fellner, M. Pierre, J. Vovelle,
\textit{About Global Existence for Quadratic Systems of Reaction-Diffusion}, J. Advanced Nonlinear Studies, \textbf{7} no. 3 (2007) pp. 491--511.

\bibitem[DF07]{DF07} L. Desvillettes, K. Fellner, \textit{Entropy Methods for 
Reaction-Diffusion Equations: Degenerate Diffusion}, Discrete and Continuous Dynamical Systems, Supplements 
Special (2007) pp. 304--312.

\bibitem[DF06]{DF06} L. Desvillettes, K. Fellner, \textit{Exponential Decay towards
Equilibrium via Entropy Methods for Reaction-Diffusion Equations}, J. Math. Anal. Appl., 
\textbf{319} (2006), pp. 157--176.

\bibitem[DM]{DM} L. Desvillettes, C. Mouhot, \textit{Large time behavior of the a priori bounds for the solutions
to the spatially homogeneous Boltzmann equation with soft potentials}, Asympt. Anal., \textbf{54} no. 3-4 (2007) 
pp. 235--245.

\bibitem[Gr\"o]{Gro92} K. Gr\"oger, \textit{Free energy estimates and
asymptotic behaviour of reaction-diffusion processes.} Preprint 20,
Institut f\"ur Angewandte Analysis und Stochastik, Berlin, 1992.

\bibitem[GGH]{GGH96} A. Glitzky, K. Gr\"oger, R. H\"unlich, \textit{Free
energy and dissipation rate for reaction-diffusion processes of
electrically charged species.} Appl. Anal. \textbf{60}, no. 3-4 (1996),
pp. 201--217.

\bibitem[GV]{GV} T. Goudon, A. Vasseur. \textit{Regularity analysis for systems of reaction- diffusion equations}, Ann. Sci. Ec. Norm. Super., (4) \textbf{43} no. 1 (2010) pp. 117--14?. 

\bibitem[HM]{HM} S.L. Hollis, J.J. Morgan. \textit{On the blow-up of solution to some semilinear and quasilinear reaction-diffusion systems}, Rocky Mountain Journal of Mathematics, \textbf{24} no. 4 (1994) pp. 1447--1465. 

\bibitem[HMP]{HMP} S.L. Hollis, J.J. Morgan, M. Pierre. \textit{Global existence and boundedness in reaction-diffusion systems}, SIAM J. Math. Anal. \textbf{18} (1987) pp. 744--761. 

\bibitem[LSU]{LSU} O.A. Ladyzenskaya, V.A. Solonnikov, N.N. Uralceva, Linear and Quasi-linear Equations of Parabolic Type, Trans. Math. Monographs, Vol. 23, Am. Math. Soc., Providence, 1968.

\bibitem[L]{JFA} D. Lamberton, \textit{Equations d'\'evolution lin\'eaires associ\'ees \`a des
semi-groupes de contraction dans les espaces $L^p$}, J. Functional Anal., \textbf{72} (1987) pp. 252--262.

\bibitem[Pie]{Pie03} M. Pierre, \textit{Weak solutions and supersolutions in $L\sp 1$ for reaction-diffusion systems}, J. Evol. Equ., \textbf{3}, no. 1 (2003) 153--168. 

\bibitem[Pie10]{Pie10} M. Pierre \textit{Global existence in reaction-diffusion systems with control of mass: a survey}. 
Milan J. Math. \textbf{78} ,no. 2 (2010), 41--455. 

\bibitem[PSch]{PS00} M. Pierre, D. Schmitt, \textit{Blowup in 
reaction-diffusion systems with dissipation of mass.} 
SIAM Rev.,  \textbf{42}  (2000)  93--106 (electronic). 

\bibitem[QS]{QS} P. Quittner, Ph. Souplet, Superlinear parabolic problems. Blow-up, global existence and steady states, Birkhauser Advanced Texts, 2007.

\bibitem[TV]{TVillani} G.Toscani, C. Villani, On the trend to equilibrium for some dissipative systems with slowly 
increasing {\sl{a priori bounds}}, J. Statist. Phys., \textbf{98} (2000) 1279-1309.
\end{thebibliography}
\end{document}